\documentclass[12pt,a4paper]{article}

\usepackage[utf8]{inputenc}
\usepackage{amsmath}
\usepackage{amsfonts}
\usepackage{amssymb}
\usepackage{amsthm}
\usepackage[colorlinks]{hyperref}
\hypersetup{
  citecolor={cyan}
}
\usepackage[left=2cm,right=2cm,top=2cm,bottom=2cm]{geometry}
\usepackage{xcolor}
\usepackage[all,cmtip]{xy}
\usepackage{tikz}
\usetikzlibrary{decorations.markings}
\usepackage{enumerate}
\usepackage{array}

\newcolumntype{C}[1]{>{\centering\arraybackslash}m{#1}}

\newcommand{\U}{\mathrm{U}}
\newcommand{\SU}{\mathrm{SU}}

\newcommand{\Hom}{\mathrm{Hom}}

\newcommand{\mft}{\mathfrak{t}}

\newcommand{\RR}{\mathbb{R}}
\newcommand{\CC}{\mathbb{C}}
\newcommand{\ZZ}{\mathbb{Z}}

\newtheorem{thm}{Theorem}[section]
\newtheorem{prop}[thm]{Proposition}
\newtheorem{cor}[thm]{Corollary}
\newtheorem{lem}[thm]{Lemma}

\theoremstyle{definition}
\newtheorem{rem}[thm]{Remark}
\newtheorem{defn}[thm]{Definition}
\newtheorem{ex}[thm]{Example}

\begin{document}

\title{GKM manifolds are not rigid}
\author{Oliver Goertsches\footnote{Philipps-Universit\"at Marburg, email:
goertsch@mathematik.uni-marburg.de}, Panagiotis Konstantis\footnote{Universit\"at zu Köln,
email: pako@mathematik.uni-koeln.de}, and Leopold
Zoller\footnote{Ludwig-Maximilians-Universit\"at M\"unchen, email: leopold.zoller@mathematik.uni-muenchen.de}}

\maketitle

\begin{abstract} We construct effective GKM $T^3$-actions with connected stabilizers on the total spaces of the two $S^2$-bundles over $S^6$ with identical GKM graphs. This shows that the GKM graph of a simply-connected integer GKM manifold with connected stabilizers does not determine its homotopy type. We complement this by a discussion of the minimality of this example: the homotopy type of integer GKM manifolds with connected stabilizers is indeed encoded in the GKM graph for smaller dimensions, lower complexity, or lower number of fixed points. Regarding geometric structures on the new example, we find an almost complex structure which is invariant under the action of a subtorus. In addition to the minimal example, we provide an analogous example where the torus actions are Hamiltonian, which disproves symplectic cohomological rigidity for Hamiltonian integer GKM manifolds.
\end{abstract}

\section{Introduction}
There is a rich and successful history of describing manifolds with torus actions through discrete
objects. Most famously this is illustrated by the bijective correspondence between toric varieties
and their fans or more specifically the correspondence between symplectic toric manifolds and Delzant
polytopes. In \cite{MR2431667} Masuda proved that the equivariant isomorphism type of a toric
manifold, when considered as a complex variety, is determined by its integral equivariant
cohomology ring (or equivalently by its GKM
graph \cite{MR4038723}). The more ambitious cohomological rigidity problem, in its original form posed by
Masuda--Suh in \cite[Problem 1]{MR2428362}, asks if a toric
manifold is determined up to (nonequivariant) homeomorphism by its integral cohomology ring.  While
the problem in this form is unsolved as of today, several variants have been considered in the
literature. For instance, when restricting attention to certain (generalized) Bott
manifolds, it was seen to be true, see \cite{MR2962979}, Sections 2.1 and 2.2 and references therein.

Motivated by the success in the toric case it is natural to look for similar results in a generalized setting.
Prominent candidates are the classes of quasitoric manifolds and torus manifolds which were also
considered in the light of the cohomology rigidity problem in \cite{MR2928539, MR2846908, MR3030690} (note that the latter fails for torus manifolds, see \cite[Example 3.4]{MR2428362}). The
main object of study in the present article is the class of so-called (integer) GKM manifolds,
named after Goresky, Kottwitz and MacPherson \cite{MR1489894}, which in particular generalizes toric manifolds. These are
compact manifolds with vanishing odd-degree (integral) cohomology, equipped with an action of
a compact torus, whose fixed point set is finite, and whose one-skeleton is a union of invariant
$2$-spheres. To such actions
one associates a labelled graph, its GKM graph, which encodes the one skeleton in a combinatoric
fashion, see Section \ref{sec:GKM}. Given a certain connectedness condition of the stabilizers, the GKM graph somewhat surprisingly encodes the entire integral equivariant
and nonequivariant cohomology as well as all characteristic classes of the manifold (see
\cite[Theorem 3.1 and Proposition 3.5]{MR4088417}). With regards to the previously mentioned rigidity problems,
the naive follow up question would be whether the GKM graph determines the (nonequivariant)
homotopy type, the homeomorphism type, or even the diffeomorphism type. Note that this kind of
question belongs in the simply-connected realm since one cannot hope to encode the fundamental
group in the GKM-graph (see \cite[Section 2]{MR3030690}) and furthermore the motivating
example of a toric manifold is also simply-connected.

The purpose of this article is to give a counterexample to these questions and to discuss edge cases, where results of the above type actually do hold. Our main result reads as follows

\begin{thm}\label{thm:einleitungmainthm}
On the total spaces of the two $S^2$-bundles over $S^6$ there exist GKM $T^3$-actions with identical GKM graph.
\end{thm}

While the two spaces in question have identical integral cohomology ring and characteristic classes
(as needs to be the case for any counterexample) they are not homotopy equivalent by Lemma
\ref{lem:pi5} below due to differing fifth homotopy groups. In particular this also shows that the
equivariant cohomological rigidity problem \cite[Theorem 1.1]{MR2431667} does not generalize to integer GKM manifolds. Regarding geometric structures, we find an almost complex structure which is invariant under a two dimensional subtorus (see Theorem \ref{thm:almostcomplex}):
\begin{thm} The total space of the nontrivial $S^2$-bundle over $S^6$ admits an almost complex structure invariant under a circle action (even an action of the two-dimensional torus) with exactly four fixed points.
\end{thm}

This implies an example of a simply connected $8$-manifold with an invariant almost complex circle
action with four isolated fixed points which is not diffeormophic to $S^2 \times S^6$. In
\cite{MR3113861} Kustarev constructs many non-simply-connected $6$-manifolds with almost complex
circle actions fixing exactly two points. Crossing these examples with $S^2$ one obtains
non-simply-connected $8$-manifolds each endowed with an almost complex circle action with four
isolated fixed points. We remark that the significance of our example is due to the vanishing of
the odd integer cohomology. This forces that each even Betti number must be equal to one (remember
that the Euler characteristic must be equal to the number of isolated fixed points). As a further
remark we note that by \cite{2001.10699v1} our example is unitary cobordant to $S^2 \times S^6$
furnished with its standard almost complex structure.

While the above examples are of course not symplectic, slight modifications give rise to the
following

\begin{thm}
On the total spaces of two $\mathbb{C}P^1$-bundles over $\mathbb{C}P^3$ which are not homotopy equivalent, there exist Hamiltonian GKM $T^3$-actions with identical x-ray.
\end{thm}

The notion of x-ray encodes in particular the GKM graph and is recalled in Section \ref{sec:symplecticstructures}. As a consequence, we show in  Proposition \ref{prop:nosymplecticrigidity} that symplectic cohomological rigidity, as defined in \cite{2002.12434v1}, does not hold for Hamiltonian integer GKM manifolds.


We also discuss positive answers to the rigidity question in special cases with respect to the
number of fixed points, dimension, and complexity of the action, where the complexity of
a $T^r$-action on $M^{2n}$ is defined as $n-r$. The situation for simply-connected integer GKM
manifolds with connected stabilizers, as known to the authors, is depicted in the table below. In
particular this shows that the counterexample from Theorem \ref{thm:einleitungmainthm} is simultaneously minimal with respect to all three
of the above parameters. Details are discussed in Section \ref{sec:minimality}.

Of the results in the table, only rows 4 and 6 are original to this paper. The first row follows
from \cite[Theorem p.\ 537]{MR268911} (alternatively, the arguments in \cite[Section 6]{MR3030690}
for $6$-dimensional torus manifolds are also applicable). The second row is a consequence of
\cite[Theorem 3.1]{MR4088417}.  The third row makes use of the generalized Poincar\'e conjecture
\cite{MR137124} and is, together with the fourth row, discussed at the end of Section \ref{sec:minimality}. Note that we do not
know if any exotic sphere admits an action of GKM type. In the remaining fifth row the first two
columns are deduced from \cite[Theorem
3.4]{MR3355120} and the third follows from \cite[Theorem 4.1]{MR3030690} (for the latter we remark
  that equivariant homeomorphisms preserve the GKM graph up to isomorphism).

\begin{center}
\begin{tabular}{c|cC{3cm}cC{4cm}}
type: & homotopy & homeomorphism & diffeomorphism & equivariant homotopy/homeo/diffeo\\ \hline
$\dim\leq 4$ &yes & yes &yes & yes\\
$\dim\leq 6$  &yes & yes& yes& ?\\
$2$ fixed points &yes & yes & ?& ?\\
$3$ fixed points &yes & yes & yes& ?\\
complexity $0$ &yes& yes& no& ?\\
else& no & no & no & no
\end{tabular}
\end{center}

\paragraph{Acknowledgements.} We are grateful
to Yael Karshon and Susan Tolman for valuable remarks on a previous
version of the paper, including a simplified construction of the action on the bundle $E$ in Section \ref{sec:construction}. We wish to thank Michael Wiemeler for several helpful comments.
This work is part of a project funded by the Deutsche Forschungsgemeinschaft (DFG, German Research Foundation) - 452427095.

\section{GKM manifolds}\label{sec:GKM}

In this paper we consider actions of compact tori $T$ on closed manifolds $M$. Given such an action, we denote its fixed point set by $M^T$, and its one-skeleton by $M_1=\{p\in M\mid \dim Tp\leq 1\}$.
\begin{defn}[\cite{MR1489894}]
We say that a $T$-action on a closed manifold $M$ is \emph{(integer) GKM} if
\begin{enumerate}
\item $H^{{\mathrm{odd}}}(M,\ZZ)=0$
\item $M^T$ is finite
\item $M_1$ is a finite union of $T$-invariant $2$-spheres.
\end{enumerate}
\end{defn}
For any GKM $T$-action on $M$, the quotient $M_1/T$ is homeomorphic to a graph, with one vertex for each fixed point, and one edge for each invariant $2$-sphere.  The tangent space of any such sphere in each of its two fixed points is an invariant subspace of the respective isotropy representation. We attach as a label to its edge in the graph the corresponding weight, considered as an element in $\ZZ_{\mft}^*/\pm 1$, where $\mft$ is the Lie algebra of $T$ and $\ZZ_{\mft}^*\subset \mft^*$ its weight lattice. This labelled graph will be called the \emph{GKM graph} of $M$.

If $M$ is equipped with a $T$-invariant almost complex structure, then the above-mentioned weights are naturally elements of $\ZZ_\mft^*$. In this setting, we attach instead to any oriented edge the corresponding weight at the initial vertex, and sometimes speak about the \emph{signed GKM graph} of the action.

\begin{ex} Any toric symplectic manifold is of GKM type. In this case, the GKM graph is given by the one-skeleton of the momentum polytope; the labels of the edges are given by the primitive vectors pointing in direction of its slopes.
\end{ex}
The (signed) GKM graph of an action is of relevance because under a certain connectedness assumption on the isotropy groups it encodes various topological properties of $M$ and the action, such as its (equivariant) cohomology algebra \cite[Corollary 2.2]{MR2790824}, and its (equivariant) characteristic classes \cite[Proposition 3.5]{MR4088417}.

\section{The minimal example}\label{sec:construction}

We construct two simply-connected GKM $T^3$-manifolds in dimension $8$ which are not (non-equivariantly) homotopy equivalent but have the same GKM graph. The underlying manifolds will be $S^2\times S^6$ as well as the total space of the non-trivial $S^2$-bundle over $S^6$.

\subsection{The construction}

We begin by recalling the construction and show that the total spaces of the two $S^2$-bundles over $S^6$ are indeed not homotopy equivalent. Principal $\SU(2)$-bundles over $S^6$ are given as pullbacks of the universal bundle $\SU(2)\rightarrow S^\infty\rightarrow \mathbb{H}P^\infty$ along maps $S^6\rightarrow \mathbb{H}P^\infty$. The space $\mathbb{H}P^\infty$ has a CW-structure with cells only in dimensions which are multiples of $4$. In particular its $7$-skeleton is $\mathbb{H}P^1=S^4$ and thus $S^4\hookrightarrow \mathbb{H}P^\infty$ induces an isomorphism on homotopy groups in dimensions up to $6$. We now use the nontrivial generator $f\colon S^6\rightarrow S^4$ of $\pi_6(S^4)\cong \mathbb{Z}_2$ to pull back the universal bundle. The resulting commutative diagram

\[\xymatrix{
\SU(2)\ar[r]\ar[d] &P\ar[r]\ar[d] & S^6\ar[d]\\
\SU(2)\ar[r] & S^\infty \ar[r]& \mathbb{H}P^\infty
}\]
induces a commutative diagram

\[\xymatrix{\cdots\ar[r]& \pi_6(S^6)\ar[r]\ar[d]^{f_*}& \pi_5(\SU(2)) \ar[r]\ar[d]^{\mathrm{id}}& \pi_5(P)\ar[r]\ar[d]& 0\ar[r]\ar[d]& \cdots\\
\cdots\ar[r]& \pi_6(\mathbb{H}P^\infty)\ar[r] & \pi_5(\SU(2))\ar[r]& 0 \ar[r]& \pi_5(\mathbb{H}P^\infty)\ar[r]&\cdots}\]
in which the rows are the long exact homotopy sequences of the bundles in question. By the choice of $f$, the map $f_*\colon \pi_6(S^6)\rightarrow \pi_6(\mathbb{H}P^\infty)=\pi_6(S^4)$ is surjective. Furthermore, the connecting homomorphism $\pi_6(\mathbb{H}P^\infty)\rightarrow \pi_5(\SU(2))$ is an isomorphism since $S^\infty$ is contractible. As a consequence the connecting homomorphism $\pi_6(S^6)\rightarrow \pi_5(\SU(2))$ is surjective. In combination with $\pi_5(S^6)=0$ this implies $\pi_5(P)=0$. Finally, we obtain a $\mathbb{C}P^1$-bundle $X\rightarrow S^6$ from $P$ by factoring out the fiberwise action of the circle $\mathrm{S}(\U(1)\times\U(1))$. Since this action is free we obtain a fiber bundle $S^1\rightarrow P\rightarrow X$. Again, via the long exact homotopy sequence, this implies the following well known

\begin{lem}\label{lem:pi5}
We have $\pi_5(X)=0$ whereas $\pi_5(S^2\times S^6)=\pi_5(S^2)=\mathbb{Z}_2$.
\end{lem}

We construct a GKM action on $X$. To do this, we show that the map $f\colon S^6\rightarrow S^4$ in the construction above can be chosen to be $T^3$-equivariant with respect to certain actions. Then we lift the action on $S^4$ to an action of the restricted universal bundle which descends to the projectivization. In this way $X$ will be an equivariant pullback and will naturally carry a $T^3$-action.

Consider the $T^3$-action on $S^5\subset \mathbb{C}^3$ given by $(s,t,u)\cdot(v,w,z)=(su,tw,uz)$. The subcircles $K_1=\{(s,s,1)\in T^3\}$, $K_2=\{(s,1,s)\in T^3\}$, and $K_3=\{(s,1,1)\in T^3\}$ generate $T^3$.

\begin{lem}\label{lem:S5nachS4}
The orbit space $S^5/K_1$ is homeomorphic to $S^4\subset \mathbb{C}\oplus\mathbb{R}\oplus \mathbb{C}$ in a way such that the induced action of $K_2$ corresponds to $(s,1,s)\cdot(v,h,w)=(sv,h,sw)$ and $K_3$ acts as $(s,1,1)\cdot (v,h,w)=(sv,h,w)$ for $(v,h,w)\in S^4$.
\end{lem}

\begin{proof}
Recall the suspension homeomorphism $\Sigma S^n\rightarrow S^{n+1}$ given by \[[(a_0,\ldots,a_n),t]\mapsto (\varphi(t)a_0,\ldots,\varphi(t)a_n,t),\]
where we set $\varphi(t)=\sqrt{1- t^2}$. Applying this twice, we see that the $K_1$-action corresponds to the doubly suspended diagonal action on $\Sigma^2 S^3$. In particular its orbit space is $\Sigma^2S^2=S^4$ as claimed. To prove the statement on the actions we need to recall the explicit form of the homeomorphism $S^3/S^1=\mathbb{C}P^1\cong S^2$.

Let $S^2_+=\{(v,h)\in S^2\subset \mathbb{C}\oplus \mathbb{R}~|~ h\geq 0\}$ be the upper hemisphere and denote by $A\subset S^2_+$ the equator. Then $S^2_+/A$ is homeomorphic to $\mathbb{C}P^1$ via $(v,h)\mapsto [v,h]$. On the other hand $S^2_+/A$ is homeomorphic to $S^2$ by the stretching map $(v,h)\mapsto (\alpha(h)v,2h-1)$, where $\alpha(h)\in [0,1]$ is a suitable scaling factor. Finally, note that for ${z}\in \mathbb{C}$ with $|z|\leq 1$ and $t=\mathrm{Im}(z)$, $s={\varphi(t)}^{-1}{\mathrm{Re(z)}}$, we have $\varphi(s)\varphi(t)=\varphi(\vert z\vert)$. Thus, under the double suspension map, when explicitly defined as above, the preimage of a point $(v,w,z)\in S^5$ is given by $\left[({\varphi(\vert z\vert)}^{-1}v,{\varphi(\vert z\vert)}^{-1}w),s,t\right]$ with $s,t$ as above. Piecing everything together we obtain
\begin{align*}
S^5\cong \Sigma^2S^3\rightarrow\Sigma^2\mathbb{C}P^1\cong \Sigma^2 S^2_+/A \cong \Sigma^2 S^2\cong S^4\subset \mathbb{C}\oplus\mathbb{R}\oplus \mathbb{C}
\end{align*}
given by
\begin{alignat*}{2}
(v,w,z)&\mapsto  [(\varphi(|z|)^{-1}v,\varphi(|z|)^{-1}w),s,t] && \in \Sigma^2S^3 \\
& \mapsto
[(\varphi(|z|)^{-1} \beta(w) v,\varphi(|z|)^{-1}|w|),s,t] && \in  \Sigma^2 S^2_+/A  \\
&\mapsto [\left(\alpha \left(\varphi(|z|)^{-1}|w|\right)\varphi(|z|)^{-1}\beta(w) v,2\varphi(|z|)^{-1}|w|-1\right),s,t]\quad && \in \Sigma^2 S^2 \\
&\mapsto \left(\alpha\left(\varphi(\vert z\vert)^{-1}\vert w\vert\right)\beta(w) v,2\vert w\vert-\varphi(\vert z\vert),z\right)&& \in S^4 ,
\end{alignat*}
where $\beta(w)$ is $1$ if $w=0$ and $\overline{w}/\vert w\vert$ otherwise. Since $\vert z\vert$ is invariant under multiplication with $S^1$, we see that the identification $S^5/K_1\cong S^4$ above is equivariant with respect to $S^1$-multiplication in the first and third coordinate.
\end{proof}

We now consider the suspended $T^3$-action on $S^6$.

\begin{prop}\label{prop:S6nachS4}
The orbit space $S^6/(K_1K_2)$ is homeomorphic to $S^4$ in a way that the orbit map $f:S^6\to S^6(K_1K_2)\cong S^4$ is $T^3$-equivariant with respect to the $T^3$-action on $S^4$ in which $K_1$ and $K_2$ act trivially, while $K_3$ acts via $(s,1,1)\cdot (v,w,h)=(sv,w,h)$. Furthermore, $f$ defines a generator of $\pi_6(S^4)$.
\end{prop}

\begin{proof}
As we have seen in the previous lemma, the projection $S^5\rightarrow S^5/K_1\cong S^4$ corresponds
to the double suspension $\Sigma^2 g$ of the Hopf map $g\colon S^3\rightarrow S^2$. It also follows from the lemma
that $S^4\rightarrow S^4/K_2\cong \Sigma \mathbb{C}P^1\cong \Sigma S^2\cong S^3$ can be
identified with $\Sigma g$ and can be chosen $K_3$-equivariant with respect to the action $(s,1,1)\cdot (v,w)=(sv,w)$ on
$S^3$. Thus $S^5\rightarrow S^5/(K_1K_2)\cong S^3$ is the composition $\Sigma g\circ \Sigma^2 g$
which is known to give a generator of $\pi_5(S^3)$, cf. \cite[p. 475]{MR1867354}. The suspension is
a generator of $\pi_6(S^4)$ by the Freudenthal suspension theorem and satisfies the equivariance
property as claimed in the lemma.
\end{proof}

Consider the quaternionic Hopf bundle
\[S^3\longrightarrow S^7\longrightarrow S^4\]
which is given by dividing out the quaternionic diagonal action by right multiplication on $S^7\subset \mathbb{H}^2$ and identifying the orbit space $\mathbb{H}P^1$ with $S^4$. The orbit map factors as \[S^7\rightarrow \mathbb{C}P^3\rightarrow \mathbb{H}P^1\cong S^4\] where the first map divides by complex diagonal multiplication from the right (viewing $\mathbb{C}$ as a subset of $\mathbb{H}$). The second map is a fiber bundle with fiber $\mathbb{C}P^1$, and as the quaternionic Hopf bundle is the restriction of the universal $\SU(2)$-bundle $S^\infty\to \mathbb{H}P^\infty$ to $\mathbb{H}P^1$, by construction $X$ is the pullback of $\mathbb{C}P^3\rightarrow S^4$ along $f$. We will now construct compatible actions on $\mathbb{C}P^3\rightarrow S^4$ which pull back to $X$.

On $S^7 \subset\mathbb{H}^2$ consider the $S^1$-action from the left defined by \[s\cdot (p,q)=(sp,sq)\]
where $p,q\in \mathbb{H}$. This commutes with the actions of the complex and quaternionic diagonals from the right. In particular $\mathbb{C}P^3\rightarrow S^4$ is equivariant with respect to the induced circle action.
\begin{prop}\label{prop:s1actionons4}
On $S^4\subset \mathbb{C}^2\oplus \mathbb{R}$ this action can be identified with $s\cdot(v,w,h)=(s^2v,w,h)$. On the fibers of $\mathbb{C}P^3\rightarrow S^4$ over the fixed points of $S^4$, the induced action is equivariantly diffeomorphic to the rotation $s\cdot (v,h)=(s^2v,h)$ on $S^2\subset \mathbb{C}\oplus \mathbb{R}$.
\end{prop}
\begin{proof}
We recall the identification $\mathbb{H}P^1\cong S^4$. Denote by $S^4_+$ the hemisphere $\{(q,h)\in S^4\subset \mathbb{H}\oplus \mathbb{R}\mid h\geq 0\}$ and by $A$ the equator at height $0$. Now the map $(q,h)\mapsto [q: h]$ induces a homeomorphism $S_+^4/A\cong \mathbb{H}P^1$. For $h\in\mathbb{R}$ we have $s\cdot [q:h]=[sq:sh]=[sqs^{-1}:h]$. The conjugation action $s\cdot q=sqs^{-1}$ leaves the complex plane fixed and rotates the $j$-$k$-plane at double speed. Thus with the correct identification $S_+^4/A\cong S^4$ this is indeed the action described in the proposition. The fixed points are the points $[v:w]\in \mathbb{H}P^1$ with $v,w\in \mathbb{C}$.
For $v,w\in \mathbb{C}$, not both zero, the map $q\mapsto [vq:wq]_\mathbb{C}$ induces a diffeomorphism from $S^3/S^1$ onto the fiber over $[v:w]$, where the quotient $S^3/S^1$ is by quaternionic multiplication with complex numbers from the right, and the notation $[\cdot:\cdot]_\mathbb{C}$ describes points in the quotient of $S^7$ by the complex diagonal action from the right. Using that $v,w\in \mathbb{C}$, the fiber inclusion $S^3/S^1\rightarrow \mathbb{C}P^3$ satisfies
\[ sqS^1\mapsto [vsq: wsq]_\mathbb{C}=s\cdot[vq:wq]_\mathbb{C},\]
and is thus equivariant with respect to quaternionic multiplication with $S^1$ from the left on $S^3/S^1$. This can be identified with the action from the proposition.
\end{proof}
The $S^1$-action on the bundle $\mathbb{C}P^3\rightarrow S^4$ has kernel generated by $-1$. We divide by the corresponding subgroup to make the action effective. Then we pull back the $S^1$-actions along the homomorphism $\psi\colon T^3\rightarrow S^1$, $(s,t,u)\rightarrow st^{-1}u^{-1}$. By Proposition \ref{prop:S6nachS4} the resulting $T^3$-action on $S^4$ is the one with respect to which the map $f\colon S^6\rightarrow S^4$ is equivariant. The space $X$ is thus the pullback of the $T^3$-equivariant bundle $\mathbb{C}P^3\rightarrow S^4$ along the $T^3$-equivariant map $f\colon S^6\rightarrow S^4$ and thus inherits a $T^3$-action.

For later purposes we will need

\begin{rem}
Since the total space of $\CC P^{3} \to S^4$ is obtained by dividing out
  a subcircle action of $\mathrm{SU}(2)$, we have that the $\CC P^{1}$-bundle $\CC P^{3} \to S^4$
  has structure group $\textrm{SU}(2)$ inherited from the principal $\textrm{SU}(2)$-bundle $S^7 \to
  S^4$. Therefore the structure group of the pullback $\CC P^{1}$-bundle $X \to S^6$ is $\SU(2)$ as well.

\end{rem}

\begin{rem}\label{rem:structure group}
The map $f$ may be approximated $T^3$-equivariantly by a smooth map. This does not affect the pullback, so we obtain a smooth structure on $X$ with respect to which the $T^3$-action is smooth.
\end{rem}

\begin{thm}\label{thm:non-rigidity}
The $1$-skeleton of the $T^3$-action on $X$ is $T^3$-equivariantly homeomorphic to the $1$-skeleton of the product $T^3$-action on $S^2\times S^6$ which acts on $S^2$ via $(s,t,u)\cdot(v,h)=(st^{-1}u^{-1}v,h)$ and on $S^6$ via $(s,t,u)\cdot(v,w,z,h)=(sv,tw,uz,h)$ for $v,w,z\in\mathbb{C}$, $h\in\mathbb{R}$.
\end{thm}

This has the following immediate
\begin{cor}
On $X$ and $S^2\times S^6$ there exist GKM $T^3$-actions with isomorphic GKM graphs.
\end{cor}

The corresponding GKM graph is given by
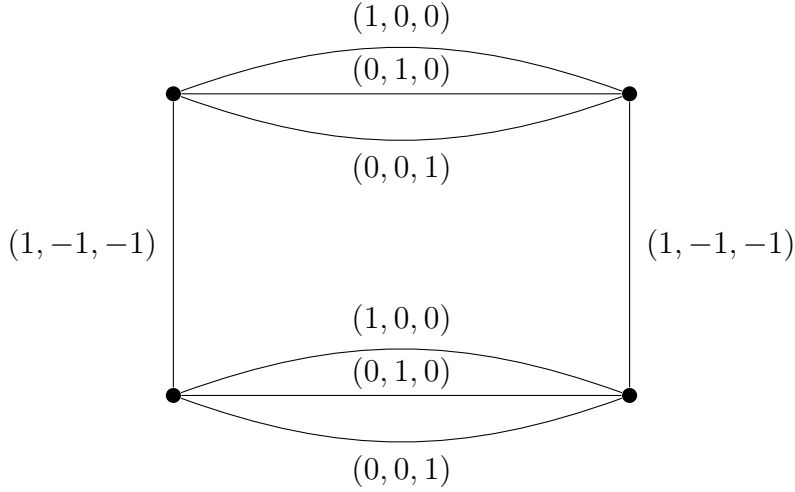
\begin{center}
\begin{tikzpicture}

\node (a) at (0,0)[circle,fill,inner sep=2pt] {};
\node (b) at (6,0)[circle,fill,inner sep=2pt]{};
\node at (3,1) {$(1,0,0)$};
\node at (3,0.3) {$(0,1,0)$};
\node at (3,-1) {$(0,0,1)$};

\draw (a) to[very thick, in=160, out=20] (b);
\draw (a) to[very thick] (b);
\draw (a) to[very thick, in=200, out=-20] (b);

\node (c) at (0,4)[circle,fill,inner sep=2pt] {};
\node (d) at (6,4)[circle,fill,inner sep=2pt]{};
\node at (3,5) {$(1,0,0)$};
\node at (3,4.3) {$(0,1,0)$};
\node at (3,3) {$(0,0,1)$};

\draw (c) to[very thick, in=160, out=20] (d);
\draw (c) to[very thick] (d);
\draw (c) to[very thick, in=200, out=-20] (d);

\draw (c) to[very thick] (a);
\draw (d) to[very thick] (b);

\node at (7.2,2) {$(1,-1,-1)$};
\node at (-1.2,2) {$(1,-1,-1)$};

\end{tikzpicture}
\end{center}

\begin{proof}[Proof of Theorem \ref{thm:non-rigidity}.]
Denote by $A\subset S^6$ the $1$-skeleton of the $T^3$-action. The $1$-skeleton of $X$ is certainly contained in $p^{-1}(A)$, where $p\colon X\rightarrow S^6$ is the projection. By naturality of the pullback $p^{-1}(A)$ is the pullback of $\mathbb{C}P^3\rightarrow S^4$ along $f|_A$. Recall that $f$ was defined as the suspension of the orbit map $S^5\rightarrow S^5/(K_1K_2)$. The space $A$ is the suspension of the $1$-skeleton $B\subset S^5$ which consists of the three disjoint circles containing those elements $(v,w,z)\in S^5$  where two coordinates are zero. As derived in Lemma \ref{lem:S5nachS4}, the map $S^5\rightarrow S^5/K_1\cong S^4\subset \mathbb{C}\oplus\mathbb{R}\oplus\mathbb{C}$ can be explicitly described as
\[(v,w,z)\mapsto \left(\alpha\left(\varphi(\vert z\vert)^{-1}\vert w\vert\right)\beta(w) v,2\vert w\vert-\varphi(\vert z\vert),z\right).\]
From this it follows that under $S^5\rightarrow S^5/(K_1K_2)\cong S^3$, the space $B$ maps to three distinct points which are fixed by $K_3$. Thus the suspended map sends $A$ to three joined lines in $(S^4)^{K_3}$. From the description of the $K_3$-action (see also Proposition \ref{prop:S6nachS4}) we deduce that $ (S^4)^{K_3}=(S^4)^{T^3}\cong S^2$. As $f|_A$ is not surjective onto this $S^2$ it is homotopic within $S^2$ to the constant map at some point $x\in S^2$. This homotopy is automatically equivariant because $S^2$ is fixed under $T^3$. By homotopy invariance of the pullback, $p^{-1}(A)$ is $T^3$-equivariantly isomorphic to the diagonal action on $A\times F_x$ where $F_x$ denotes the fiber of $\mathbb{C}P^3\rightarrow S^4$ over $x$.
In combination with Proposition \ref{prop:s1actionons4} it follows that the 1-skeleton of $X$ has the desired form.
\end{proof}

\subsection{Almost complex structures}\label{subsec:acs}
The manifold $X$ does not admit a $T^3$-invariant almost complex structure as no structure of a signed GKM graph that is compatible with the GKM graph of $X$ can admit a connection in the sense of \cite{MR1823050} (to see this, note that
transport along a horizontal edge via a connection can not map the other two horizontal edges onto themselves or one another, so they would both need to map to the single adjacent vertical edge). However we now argue that there is an almost complex structure on $X$ that is invariant under a $2$-dimensional subtorus. We want to stress that this restricted action will not be GKM.

 Recall that $X$ is a $T^3$-equivariant $\CC P^{1}$-bundle over $S^6$, where $T^3$ acts on $S^6$ in the standard way $(s,t,u)\,\cdot \,(v,w,z,h)=
  (sv,tw,uz,h)$ in $S^6 \subset \CC^3 \oplus \RR$.  It is well-known that $S^6$
  admits an almost complex structure that is invariant unter an action of the compact Lie group
  ${\mathrm{G}}_2$. The action of a maximal torus $T^2\subset {\mathrm{G}}_2$ on $S^6$ can be
  identified as a subaction of our $T^3$-action, namely as that of the subtorus
  $T^2:=\{(st,s^{-1},t)\}\subset T^3$, see e.g.\ \cite[Example 1.9.1]{MR1823050}. Now, the tangent bundle $TX$ decomposes as $TX = \pi^* TS^6 \oplus V_F$, where $V_F$ is the subbundle consisting of the tangent spaces of the fibers of $X\rightarrow S^6$. On $\pi^*TS^6$ we obtain a $T^2$-invariant almost complex structure by pulling back the $T^2$-invariant almost complex structure from $S^6$. For $V_F$ consider a fiber
  of $X$ over $S^6$ which can be identified with $\CC P^{1}$. We choose an $\textrm{SU}(2)$-invariant almost complex structure on $\CC P^{1}$, which defines an almost complex structure on the fiber. Since
  the structure group is $\textrm{SU}(2)$ this definition does not depend on the identification
  of the fiber with $\CC P^{1}$, cf. Remark \ref{rem:structure group}. This defines an
  almost complex structure on $V_F$. Moreover the $T^3$-action leaves this almost complex
  structure invariant. To see this recall that the $T^3$-action is induced by a circle action on $S^7$ which commutes
  with the action of $\textrm{SU}(2)$ of the principal bundle (see the paragraph before Proposition \ref{prop:s1actionons4}). In particular this means that
  the circle action preserves the fibers of $\CC P^{3} \to S^4$ and the transformations are given
  by elements of the structure group $\textrm{SU}(2)$. Thus the circle action respects the almost
  complex structure on the fibers and consequently so does the pullback $T^3$-action on $V_F$. Altogether one obtains
  a $T^2$-invariant almost complex structure on $X$ by restricting the $T^3$-action on $V_F$ to
  $T^2$.
Now, we arrive at the following theorem, which gives an example of an almost complex
circle action on an $8$-manifold with four isolated fixed points such that the odd cohomlogy
vanishes and which is not diffeomorphic to
$S^2 \times S^6$.

\begin{thm}\label{thm:almostcomplex} The manifold $X$, which is not diffeomorphic to $S^2\times S^6$, admits an almost complex structure invariant under a circle action (even an action of the two-dimensional torus) with exactly four fixed points.
\end{thm}

\section{Minimality of the example}\label{sec:minimality}

In this section we observe that our example is optimal with regards to three properties: complexity, dimension, and number of fixed points.\\

\noindent {\bf Complexity.} For an effective $T^r$-action with non-empty fixed point set on a compact $2n$-dimensional manifold one has $r\leq n$ and the number $n-r$ is called the complexity of the action. Thus the previously constructed $T^3$-action on $X$ is of complexity $1$. Complexity $0$ manifolds of the above type are also known as torus manifolds. Theorem 3.4 in \cite{MR3355120} states that
the nonequivariant homeomorphism type of a simply-connected torus manifold $M$ with
$H^{{\mathrm{odd}}}(M,\ZZ)=0$ is encoded in the face poset of its orbit space together with
a function that associates to each face the corresponding isotropy group. If we denote by $T$ the
acting torus, then the face poset is -- as a partially ordered set -- equivalent to the closed orbit type stratification $\chi$ which is the collection of all connected components of $M^H$ where $H$ runs through all subgroups of $T$, ordered by inclusion. On this poset
one considers the function that associates the respective principal isotropy group to a connected component of
$M^H$. We stress that the homeomorphism type of $M$ is determined just by the combinatorial
data of the poset together with the function which associates the isotropy groups and does not
require specific knowledge of the isotropy submanifolds.

For GKM manifolds this information is encoded in the GKM graph: Let $M$ be a GKM $T$-manifold (not necessarily of complexity $0$) with GKM graph $\Gamma$. For a subgroup $H\subset T$, set $\Gamma^H\subset \Gamma$ to be the minimal subgraph that contains all edges whose weights -- understood as elements in $\Hom(T\rightarrow S^1)$ -- vanish on $H$. Define $\chi_\Gamma$ to be the set consisting of all connected components of subgraphs of the form $\Gamma^H$, partially ordered by inclusion.

\begin{lem}\label{lem:stratification}
There is bijection $\chi\rightarrow \chi_\Gamma$ which preserves the partial ordering such that for some $N\in \chi$ its principal isotropy group is given by the intersection of all kernels of the weights of all edges emanating from a single vertex in the corresponding element of $\chi_\Gamma$.
\end{lem}

\begin{proof}  The one-skeleton of
$M^H$ is encoded in the minimal (possibly disconnected) subgraph $\Gamma^H\subset \Gamma$. As the odd cohomology of (every component of) $M^H$ vanishes, see \cite[Lemma 2.2]{MR2283418}, the $T$-action on every component of $M^H$ is again GKM and, since GKM graphs are always connected, the connected components of $\Gamma^H$ are the GKM
graphs of the components of $M^H$. The principal isotropy type on a component of $M^H$ can be reconstructed in the way described in the lemma, since the weights at a fixed point determine the action on a neighbourhood. This also implies the injectivity of the correspondence: if $N\subset M^H$ and $N'\subset M^{H'}$ are connected components corresponding to the same element of $\chi_\Gamma$ then the principal isotropy groups of $N$ and $N'$ agree and are equal to some $U$ containing both $H$ and $H'$. Thus both are components of $M^U$ and have nonempty intersection, which implies $N=N'$.
\end{proof}

\noindent {\bf Dimension.} The main theorem in \cite{2003.11298v1} implies that for a simply-connected $6$-dimensional integer GKM manifold $M$ whose stabilizer groups satisfy a certain assumption which is satisfied if they are connected, the GKM graph encodes the nonequivariant diffeomorphism type of $M$.
In dimension $4$, even the equivariant diffeomorphism type is known to be encoded by \cite[Theorem p.\ 537]{MR268911} (or by the arguments for $6$-dimensional torus manifolds in \cite[Section 6]{MR3030690}).
Thus, our $8$-dimensional examples of GKM manifolds that are not homotopy equivalent but have the same GKM graph have the lowest possible dimension with this property.\\

\noindent {\bf Number of fixed points.} Our example has exactly four fixed points. We would like to argue that this is the minimal possible number of fixed points in our situation. For a simply-connected integer GKM manifold of positive dimension we always have at least two fixed points. In case of exactly two fixed points, the manifold has the integer homology of a sphere. In this situation, by collapsing the complement of a disc, one always finds a map to the standard sphere that induces an isomorphism in integer homology, which by the homology Whitehead theorem is a homotopy equivalence. From the (topological) generalized Poincar\'e
conjecture \cite{MR137124} we deduce that the manifold has to be homeomorphic to the standard sphere.

For three fixed points we have the following proposition:

\begin{prop}For a simply-connected integer GKM manifold $M$ in arbitrary dimension, with three fixed points, the GKM graph determines the nonequivariant diffeomorphism type of $M$.
\end{prop}
\begin{proof}

  If the action has precisely three fixed points, then the integral homology of $M$ is $\ZZ$ in
  degrees $0$, $n$, and $2n$, where $2n$ is the dimension of $M$, and $0$ in all other degrees. In
  this case, by Corollary $B$ of \cite{MR2355782} the diffeomorphism type of $M$ is determined by the
  Pontryagin number $p_n^2(M)[M]\in \ZZ$. Note that the cohomology ring of $M$ is given by
  $\ZZ[x]/\langle x^3  \rangle$ with $x \in H^n(M)$, thus they have a natural orientation
  given by $x^2  = (-x)^2$. Hence the Pontryagin number is
  encoded in the GKM graph of $M$ by \cite{MR4088417}.
\end{proof}

\section{The Hamiltonian example}

We will construct a Hamiltonian $T^3$-action of GKM type on a space $Y$ which has the signed GKM graph of a diagonal action on $\mathbb{C}P^1\times \mathbb{C}P^3$, while $Y$ is not homotopy equivalent to the latter product.

\subsection{The construction}

The $T^3$-space $Y$ arises as an equivariant pullback of the bundle $X\rightarrow S^6$ along a map $k\colon \mathbb{C}P^3\rightarrow S^6$. Explicitly we define $k$ as the collapsing map $\mathbb{C}P^3\rightarrow \mathbb{C}P^3/\mathbb{C}P^2\cong S^6$, where $\mathbb{C}P^2$ is embedded as those points $[z_0:\ldots:z_3]\in \mathbb{C}P^3$ with $z_3=0$. The identification of the quotient space with $S^6$ can be done such that $k$ is equivariant with respect to the action \[(s,t,u)\cdot [z_0:z_1:z_2:z_3]=[sz_0:tz_1:uz_2:z_3]\] on $\mathbb{C}P^3$ and the standard $T^3$-action on $S^6\subset \mathbb{C}^3\oplus\mathbb{R}$. Explicitly, let $S^6_+=\{(v,w,z,h)\in S^6~|~ h\geq 0\}$ be the upper hemisphere and $A\subset S^6_+$ the equator. Then one has homeomorphisms \[\mathbb{C}P^3/\mathbb{C}P^2\leftarrow S^6_+/A\rightarrow S^6,\] where the first map is defined by $(v,w,z,h)\mapsto [v:w:z:h]$ and the second map is defined by stretching along the real coordinate as in the proof of Lemma \ref{lem:S5nachS4}.

We define $Y$ as the pullback of the $T^3$-equivariant bundle $X\to S^6$ along $k$, which is the same as the pullback of the $S^1$-quotient $\mathbb{C}P^3\to S^4$ of the quaternionic Hopf bundle $S^7\to S^4$ along $f\circ k$, where $f$ is defined in Proposition \ref{prop:S6nachS4}. Since $k$ maps the one-skeleton of $\mathbb{C}P^3$ to the one-skeleton of $S^6$ and the restriction of $f$ to the one-skeleton of $S^6$ was shown to be equivariantly homotopic to a constant map, the same holds for the composition $f\circ k$ when restricted to the one skeleton of $\mathbb{C}P^3$. Analogously to Theorem \ref{thm:non-rigidity} we obtain

\begin{thm}\label{thm:non-rigidity-symplectic}
The $1$-skeleton of the $T^3$-action on $Y$ is $T^3$-equivariantly homeomorphic to the $1$-skeleton of the product $T^3$-action on $\mathbb{C}P^1\times \mathbb{C}P^3$ which acts on $\mathbb{C}P^1$ via $(s,t,u)\cdot [v,w]=[st^{-1}u^{-1}v,w]$ and on $\mathbb{C}P^3$ via $(s,t,u)\cdot [z_0:z_1:z_2:z_3]=[sz_0:tz_1:uz_2:z_3]$.
\end{thm}

It remains to prove that $Y$ is not homotopy equivalent to the product $\mathbb{C}P^1\times \mathbb{C}P^3$.
While the strategy is largely analogous to that of Lemma \ref{lem:pi5}, the details are slightly more involved, drawing from classical results on homotopy groups of spheres.

\begin{lem}\label{lem:fancyhomotopystuff}
Let $g$ be the Hopf map $S^3\rightarrow S^2$ and $i\colon S^4\cong \mathbb{H}P^1\rightarrow \mathbb{H}P^\infty$ be the inclusion. Then $i\circ \Sigma^2 g\circ \Sigma^3 g\circ \Sigma^4 g$ defines a non-trivial element of $\pi_7(\mathbb{H}P^2)$.
\end{lem}

\begin{proof}
As observed in Proposition \ref{prop:S6nachS4}, the map $\Sigma g\circ \Sigma^2 g$ is a generator of $\pi_5(S^3)$. In particular it induces a non-trivial map $\pi_5(S^5)\rightarrow \pi_5(S^3)$. Since $g$ is the projection of an $S^1$-bundle, it induces isomorphisms on higher homotopy groups, so the composition $g\circ \Sigma g\circ \Sigma^2g$ induces a non-trivial map on $\pi_5$. Hence it is a generator of $\pi_5(S^2)$.

The next step is to show that its two-fold suspension $\Sigma^2 g\circ \Sigma^3 g\circ \Sigma^4 g$
  is still non-trivial. This follows from the EHP-sequence \cite{MR83124} (or \cite{MR1867354} for
  a modern exposition)
  which involves the suspension homomorphism or rather its localization at $2$ in a long exact
  sequence. The part of the sequence we are interested in reads
\[
  \cdots\rightarrow \pi_5(S^2)_{(2)}\xrightarrow{E_{(2)}}\pi_6{(S^3)}_{(2)}\rightarrow \pi_6(S^5)_{(2)}\rightarrow \pi_4(S^2)_{(2)}\xrightarrow{E_{(2)}} \pi_5(S^3)_{(2)}\rightarrow\cdots
\]
where $E$ denotes the suspension homomorphism and the index $(2)$ denotes localization at $2$.
Using that $\pi_5(S^2)=\mathbb{Z}_2$ (cf. \cite{MR37507}) and $\pi_6(S^5)=\ZZ_2$ by the
  Freudenthal suspension theorem as well as $\pi_6(S^3)=\mathbb{Z}_{12}$ (cf. \cite{MR0143217}),
  the left part of the sequence becomes
\[
  \cdots\rightarrow\mathbb{Z}_2\xrightarrow{E_{(2)}} \mathbb{Z}_4\rightarrow \mathbb{Z}_2\rightarrow 0
\]
 since $E\colon \pi_4(S^2)\rightarrow \pi_5(S^3)$ is an isomorphism (as argued in section \ref{sec:construction}, the two groups are $\mathbb{Z}_2$ and generated by $g\circ \Sigma g$ and $\Sigma g\circ\Sigma^2 g$ respectively). Thus $E\colon \pi_5(S^2)\rightarrow \pi_6(S^3)$ is injective because this holds for the localized map and $\pi_5(S^2)\rightarrow\pi_5(S^2)_{(2)}$ is an isomorphism. An analogous sequence exists for the suspension $E$ on $S^3$. For odd spheres, no localization is necessary and the relevant part of the sequence reads
\[\cdots\rightarrow\pi_6(S^3)\xrightarrow{E} \pi_7(S^4)\rightarrow \pi_7(S^7)\rightarrow\cdots\]
which is $\mathbb{Z}_{12}\rightarrow \mathbb{Z}_{12}\oplus \mathbb{Z}\rightarrow\mathbb{Z}$. Consequently the left hand map is necessarily injective. In total we deduce that the map $\Sigma^2 g\circ \Sigma^3 g\circ \Sigma^4 g$, which is the double suspension of the generator  of $\pi_5(S^2)$, defines a non-trivial element of order $2$ in $\pi_7(S^4)$.

It remains to prove that it survives the inclusion into $\mathbb{H}P^\infty$. It suffices to
consider the inclusion into the $8$-skeleton $\mathbb{H}P^2$. The latter arises from $S^4$ by
attaching a single $8$-cell via the projection in of the Hopf fibration $g'\colon
S^7\rightarrow S^4$. By homotopy excision the kernel of
$\pi_7(S^4)\rightarrow\pi_7(\mathbb{H}P^2)$ is generated by $g'$. From the long homotopy sequence
of $S^3\rightarrow S^7\rightarrow S^4$ we deduce that $[g']\in\pi_7(S^4)$ is of infinite order.
This implies that the class of $\Sigma^2 g\circ \Sigma^3 g\circ \Sigma^4 g$, which is torsion, is
not contained in the kernel of $\pi_7(S^4)\rightarrow \pi_7(\mathbb{H}P^2)$.
\end{proof}

\begin{prop}
We have $\pi_6(Y)= \mathbb{Z}_6$ whereas $\pi_6(\mathbb{C}P^1\times \mathbb{C}P^3)=\mathbb{Z}_{12}$.
\end{prop}

\begin{proof}
Denote by $Q$ the pullback of the universal $\SU(2)$-bundle along the map $ i\circ f\circ k\colon \mathbb{C}P^3\rightarrow \mathbb{H}P^\infty$, where $i,f,k$ are as above. We have a principal $S^1$-bundle $S^1\rightarrow Q\rightarrow Y$ so it suffices to compute the homotopy groups of $Q$. There is a pullback diagram

\[\xymatrix{
\SU(2)\ar[r]\ar[d] &Q\ar[r]\ar[d] & \mathbb{C}P^3\ar[d]^{i\circ f\circ k}\\
\SU(2)\ar[r] & S^\infty \ar[r]& \mathbb{H}P^\infty
}\]
which induces a commutative diagram

\[\xymatrix{\cdots\ar[r]& \pi_7(\mathbb{C}P^3)\ar[r]\ar[d]^{(i\circ f\circ k)_*}& \pi_6(\SU(2)) \ar[r]\ar[d]^{\mathrm{id}}& \pi_6(Q)\ar[r]\ar[d]& 0\ar[r]\ar[d]& \cdots\\
\cdots\ar[r]& \pi_7(\mathbb{H}P^\infty)\ar[r] & \pi_6(\SU(2))\ar[r]& 0 \ar[r]& \pi_6(\mathbb{H}P^\infty)\ar[r]&\cdots}\]
on long exact homotopy sequences. The group $\pi_7(\mathbb{C}P^3)$ is generated by the canonical
  projection $p\colon S^7\rightarrow \mathbb{C}P^3$. By \cite[Lemma 9.2]{MR270373} the composition
  $k\circ p$ is homotopic to the quadruple suspended Hopf map $\Sigma^4 g$.
   We have $f = \Sigma^2 g\circ \Sigma^3 g$ (cf. proof of Proposition \ref{prop:S6nachS4}) and thus by Lemma
   \ref{lem:fancyhomotopystuff} we conclude that $(i\circ f\circ k)_*([p])=[i\circ \Sigma^2 g\circ
   \Sigma^3 g\circ \Sigma^4 g]$ is a non-trivial element of order $2$ in
   $\pi_7(\mathbb{H}P^\infty)$. Using the fact that
   $\pi_7(\mathbb{H}P^\infty)\rightarrow\pi_6(\SU(2))$ is an isomorphism and
   $\pi_6(\SU(2))=\mathbb{Z}_{12}$ we see that $\pi_6(Y)=\mathbb{Z}_{6}$.
\end{proof}

\subsection{Symplectic structures}\label{sec:symplecticstructures}

Recall that the space $Y$ was defined as the total space of the pullback of an equivariant $\mathbb{C}P^1$-fiber bundle $\mathbb{C}P^3\to S^4$ along $f\circ k:\mathbb{C}P^3\to S^6\to S^4$. Let $\pi \colon Y \to \CC P^{3}$ be the projection map, then by
Thurston's construction \cite{MR0402764}, see also \cite[Theorem 6.1.4]{MR3674984}, $Y$ admits a sympletic structure $\omega
= \omega_F+C\pi^\ast(\omega_B)$, where
$\omega_B$ is the Fubini-Study form on $\CC P^{3}$, $\omega_F$ is a closed $2$-form which
restricts to the Fubini-Study form on the fibers (after identifying them with $\CC P^1$ up to elements of the structure group) and $C>0$ large enough. Note that
Thurston's construction is applicable since the Fubini-Study form on $\CC P^{1}$ is invariant
under the structure group of $\pi$ (cf. Remark \ref{rem:structure group}) and the inclusion of a fiber is surjective on cohomology due to the fact that the odd-dimensional cohomology of fiber and base vanishes.
As argued in Section \ref{subsec:acs}, the $T^3$-action on $Y$ preserves the fibers of $\pi$, with the transformations between fibers corresponding to elements of the structure group. The restriction of $\omega_F$ to any fiber is invariant under the structure group.
Thus if we average $\omega_F$ over
$T^3$ we obtain a $T^3$-invariant closed $2$-form $\widetilde\omega_F$ whose restriction to any fiber agrees with that of $\omega_F$.
Clearly $\omega_B$ is also invariant with respect to the
$T^3$-actions, hence $\widetilde \omega = \widetilde\omega_F + C\pi^\ast(\omega_B)$ is
a $T^3$-invariant symplectic form (potentially replacing $C$ by a larger constant) on $Y$. In this
way, a nonempty open set in the second cohomology of $Y$ can be realized as cohomology classes of
$T^3$-invariant symplectic forms on $Y$. As $Y$ is simply-connected, the $T^3$-action on $Y$ is
automatically Hamiltonian with respect to any of these symplectic forms.


\begin{rem}
It is possible to write the fiber bundle $\CC P^{3} \to S^4$ as the projectivization of
the vector bundle $E:=S^7 \times_{\SU(2)} \mathbb{C}^2 \to S^4$, where $\SU(2)$ acts on
$\mathbb{C}^2$ via the standard representation. Thus, the space $Y$ is the projectivization of
the $T^3$-equivariant complex vector bundle $(f\circ k)^*E\to \CC P^3$.
By \cite[Chapter I, \S 6]{MR2815674}, any complex vector bundle over $\CC P^3$ admits a holomorphic structure, which may be no longer equivariant, even if the original bundle was equivariant. Consequently, by \cite[Proposition 3.18]{MR2451566} its projectivization admits a K\"ahler structure. It follows that the space $Y$ is (nonequivariantly) K\"ahler. We do not know if it admits a $T^3$-invariant K\"ahler structure.
\end{rem}
In this section, we compare the symplectic structures on $Y$ and $\CC P^1\times \CC P^3$ in regard to two aspects: their x-ray and the symplectic cohomological rigidity problem \cite{2002.12434v1}.

We recall the notion of the x-ray of a Hamiltonian action of a torus $T$ on a manifold $M$ with momentum map $\mu\colon M\rightarrow \mathfrak{t}^*$: it is given by the datum of the closed orbit type stratification $\chi$ (as a partially ordered set) which consists of the connected components of all submanifolds $M^H\subset M$, where $H\subset T$ is a subgroup, together with the function that associates to $N\in \chi$ the polytope $\mu(N)\subset \mathfrak{t}^*$. Thus in the GKM case it encodes in particular the (signed) GKM graph (associated to an almost complex structure which is compatible with the symplectic form). However, the information on the lengths of the edges is lost, when passing from the x-ray to the GKM graph. We have the following addendum to Theorem \ref{thm:non-rigidity-symplectic}:

\begin{prop}
On $Y$ and $\mathbb{C}P^1\times\mathbb{C}P^3$ there exist $T^3$-invariant symplectic forms with momentum maps whose x-rays coincide.
\end{prop}

\begin{proof}
We have seen already that the one-skeletons of the two actions are equivariantly homeo\-morphic. It suffices to prove the existence of respective momentum maps which agree on the one skeleton with respect to this homeomorphism: by Lemma \ref{lem:stratification}, $\chi$ is encoded in the GKM graph. For some $N\in \chi$, the momentum image $\mu(N)$ is determined by the image under $\mu$ of all fixed points that are contained in $N$. These are however also determined by the corresponding subgraph.

In order to prove that there are momentum maps on $Y$ and $\mathbb{C}P^1\times\mathbb{C}P^3$ which coincide on the one-skeleton (with respect to a fixed homeomorphism), we start by investigating a momentum map $\mu$ of the symplectic form $\omega$ on $Y$ as constructed above. For a fixed $\omega$, $\mu$ is unique up to translation by some element in $\mathfrak{t}^*$. Let $A$ be the one skeleton of the action on $Y$ and $B$ be the one skeleton of $\mathbb{C}P^3$.
Every invariant $2$-sphere in $A$ gets mapped by $\mu$ to an affine linear segment in $\mathfrak{t}^*$. The slope of this segment, when moving from a fixed point $p$ in a sphere to the other fixed point $q$, is determined up to sign by the weight of the sphere in $\mathfrak{t}^*/\pm$. The sign is determined by the orientation on the sphere which is induced by $\omega$. The projection $\pi\colon Y\rightarrow \mathbb{C}P^3$ maps some spheres in $A$ homeomorphically onto invariant spheres in $\CC P^3$ (we call those horizontal) and is constant on the remaining spheres which are precisely the fibers over the fixed points of $\mathbb{C}P^3$ (we call those vertical). From the construction of $\omega$ we see that $\pi$ is not necessarily a symplectomorphism on horizontal spheres (where spheres in $\mathbb{C}P^3$ are equipped with the restriction of $\omega_B$). However if the constant $C$ is large enough, at least $\pi$ is orientation preserving on horizontal spheres.

By Theorem \ref{thm:non-rigidity-symplectic}, the subspace of all horizontal spheres in $A$ decomposes into two connected components $A_h^+$ and $A_h^-$, each of which gets homeomorphically mapped to $B$ in equivariant and sphere-wise orientation preserving fashion. From the specific weights and orientations we deduce that $\mu|_{A_h^\pm}$ has to agree with $\mu_B\circ \pi$ up to translation and rescaling, where $\mu_B$ is a momentum map for $\omega_B$. Now the vertical spheres get mapped to parallel line segments of slope $\pm(1,-1,-1)$ in the standard basis of $\mathfrak{t}^*$. From the previous considerations we deduce that the signs of the slope must agree when moving on vertical spheres from $A_h^+$ to $A_h^-$. Also this forces $\mu|_{A_h^+}$ and $\mu|_{A_h^-}$ to be scaled in the same way and implies that all the lengths of the segments that are the images of vertical spheres under $\mu$ agree.

To sum up the above discussion, $\mu|_A$ is determined up to translation, global rescaling as well as the length and sign of the vertical edges. The same considerations apply not only to $Y\rightarrow \mathbb{C}P^3$ but also to the trivial bundle $\mathbb{C}P^1\times\mathbb{C}P^3\rightarrow \mathbb{C}P^3$. Thus it remains to show that the above parameters can be manipulated in a way that the momentum maps agree on the one skeleton. Translation can be manipulated arbitrarily and global rescaling of an associated momentum map is achieved by rescaling the symplectic form $\omega$. The sign of $\mu$ on vertical edges can be changed by replacing $\omega=\tilde{\omega}_F+C\omega_B$ with $-\tilde{\omega}_F+C\omega_B$. Finally the length of the vertical edges is the only thing that can not be manipulated freely. However we can make it arbitrarily short when compared to the basic edges by enlarging the constant $C$. Thus we can change the symplectic forms to make the momentum maps agree on the one-skeleton.
\end{proof}

We wish to demonstrate that the pair of $Y$ and $\CC P^1\times \CC P^3$ are a counterexample to the symplectic cohomological rigidity problem for integer GKM manifolds. The symplectic variant of the cohomological rigidity problem, as posed in \cite{2002.12434v1}, asks for families of symplectic manifolds that are distinguished by their cohomology ring and the cohomology classes of their symplectic structures.
\begin{prop}\label{prop:nosymplecticrigidity}
On $Y$ and $\CC P^1\times \CC P^3$ there exist symplectic forms that are intertwined via the isomorphism on cohomology induced by the equivariant isomorphism of the respective one-skeleta from Theorem \ref{thm:non-rigidity-symplectic}.
\end{prop}
\begin{proof}
We observed at the beginning of the section that an open subset of the second cohomology of $Y$ is realized by symplectic structures on $Y$; on $\CC P^1\times \CC P^3$ the same is true for an open and dense subset of the second cohomology group. The assertion is immediate.
\end{proof}
We observe that even more is true: a closed equivariant extension of the symplectic form $\omega$ on a symplectic manifold $M$ with Hamiltonian $T$-action is an equivariant differential form of the form $\omega + \mu$, where $\mu$ is a momentum map of the $T$-action, see \cite[Example 4.16]{MR4025581}. One can choose appropriate momentum maps on $Y$ and $\CC P^1\times \CC P^3$ such that the isomorphism $H^2_{T^3}(Y) \cong H^2_{T^3}(\CC P^1\times \CC P^3)$ induced by the homeomorphism from Theorem \ref{thm:non-rigidity-symplectic} intertwines the corresponding equivariant cohomology classes. In other words, these examples are not equivariantly symplectically cohomologically rigid.

\bibliography{dim8GKMv3}
\bibliographystyle{acm}
\end{document}